\documentclass[a4paper,12pt]{amsart}
\usepackage{amsmath,amsfonts,amssymb,amsthm}
\usepackage{latexsym,graphicx}
\usepackage[matrix,arrow,ps,color,line,curve,frame]{xy}

\newtheorem{proposition}{\sc Proposition}[section]
\newtheorem{lemma}[proposition]{\sc Lemma}
\newtheorem{corollary}[proposition]{\sc Corollary}

\theoremstyle{definition}
\newtheorem{definition}[proposition]{\sc Definition}

\theoremstyle{remark}
\newtheorem{remark}[proposition]{\sc Remark}

\newcommand{\can}{\operatorname{\it can}}

\newcommand{\id}{\operatorname{id}}

\def\sw#1{{\sb{(#1)}}}

\newcommand{\ot}{\otimes}
\renewcommand{\phi}{\varphi}
\renewcommand{\epsilon}{\varepsilon}
\renewcommand{\subset}{\subseteq}

\def\o{\sp{[1]}}
\def\t{\sp{[2]}}

\newcommand{\co}{\,\mathrm{co}\,}
\def\C{{\mathbb C}}
\def\N{{\mathbb N}}

\def\Z{{\mathbb Z}}

\def\id{{\rm id}}
\def\cO{\mathcal{O}}

\newcommand{\mZt}{\Z_2}

\newcommand{\Hom}{{\text{Hom}}}
\newcommand{\rp}{{\mathbb{R}\mathbb{P}}}

\begin{document}

\title{Quantum principal bundles over quantum real projective spaces}

\author{Tomasz Brzezi\'nski}
 \address{ Department of Mathematics, Swansea University, 
  Singleton Park,  Swansea SA2 8PP, U.K.} 
  \email{T.Brzezinski@swansea.ac.uk}  
\author{Bartosz Zieli\'nski}
\address{Instytut Matematyczny, Polska Akademia Nauk, ul.~\'Sniadeckich 8, Warszawa, 00-956 Poland. \newline\indent
Department of Theoretical Physics and Informatics, University of \L{}\'od\'z, Pomorska 149/153 90-236 \L{}\'od\'z, Poland.}
\email{bzielinski@uni.lodz.pl}

\begin{abstract}
Two hierarchies of quantum principal bundles over quantum real projective spaces are constructed. One hierarchy contains bundles with $U(1)$ as a structure group, the other has the quantum group $SU_q(2)$ as a fibre. Both hierarchies are obtained by the process of prolongation from bundles with the cyclic group of order 2 as a fibre. The triviality or otherwise of these bundles is determined by using a general criterion for a prolongation of a comodule algebra to be a cleft Hopf-Galois extension.
\end{abstract}
\maketitle
\setcounter{tocdepth}{1}
\tableofcontents

\section{Introduction}
For considerable time it has been  argued that faithfully flat Hopf-Galois extensions 
 or Hopf-Galois extensions  
 admitting  strong connections 
  or  {\em principal comodule algebras} 
  should be considered as objects representing principal bundles in noncommutative geometry; see e.g.\ \cite{Sch:pri}, \cite{BrzMaj:gau}, \cite{h-pm96}, \cite{hkmz}. Prolongation of the structure group to a larger group is one  of standard methods of constructing principal bundles in classical geometry. 
   In the case of principal comodule algebras an analogous  (or dual) construction  starts with a principal comodule algebra over a Hopf algebra $\bar{H}$ and produces a principal comodule algebra over any Hopf algebra $H$ for which there exists a Hopf algebra map $H\to \bar{H}$, by using the cotensor product; see \cite[Remark~3.11]{Sch:pri}. In this paper we derive necessary and sufficient conditions for a prolongation to be trivial, and study prologations of quantum (Euclidean) spheres.  

Although this text deals primarily with prolongations we also comment on the opposite process of {\em reduction}. This is the process by which a principal comodule $H$-algebra is reduced to a principal $\bar{H}$-comodule algebra along a surjective Hopf algebra map $H\to \bar{H}$. In particular we illustrate the possibility of reducing trivial (that is smash product) principal comodule algebras to non-trivial principal comodule algebras.

The paper is organised as follows. In Section~\ref{sec.pri} we gather preliminary information about principal comodule algebras and prolongations. In Section~\ref{sec.sph.proj} we present a hierarchy (or a sequence of surjective maps) of  coordinate algebras of quantum spheres. Each algebra in this hierarchy is a principal comodule algebra of the Hopf algebra $\mathcal{O}(\mZt)$ generated by a single self-adjoint  element that squares to 1. Invariant subalgebras can be identified with quantum real projective spaces.  We spend some time presenting the algebraic structure of low dimensional quantum spheres and of the coordinate algebra of   the quantum real projective plane $\cO(\rp_q^2)$. The main results of the paper are contained in Sections~\ref{sec.cleft} and \ref{sec.exam}. In the former, a sufficient and necessary condition for prolongation of a comodule algebra to be cleft is derived. In the latter we prolong quantum spheres to principal comodule algebras of $\cO(U(1))$ and $\cO(SU_q(2))$. We use criterion derived in Section~\ref{sec.cleft} to determine which of the constructed comodule algebras are trivial, and we describe algebraic contents of  algebras obtained as prolongations. In particular, we prove that  $\cO(U(1))$-prolongations of coordinate algebras of quantum spheres $\cO(S_q^m)$, $m>1$,  are non-trivial.  Furthermore,  we give the presentation by generators and relations of  $\cO(U(1))$-prolongations of even-dimensional spheres, describe their irreducible $*$-representations and construct Fredholm modules over them. We prove that $\cO(U(1))$-prolongations of odd-dimensional spheres are isomorphic as algebras (but not as comodules) to $\cO(S^{2n+1}_q)\otimes \cO(U(1))$.  
We describe 
the smash product structure of the $\cO(SU_q(2))$-prolongation of $\cO(S_q^2)$ and   explain that, for $m=2,3$,  the 
algebras $\cO(S^{m}_q)\square_{\mathcal{O}(\mZt)}\cO(U(1))$ are non-trivial principal  $\cO(U(1))$-comodule algebras obtained by reduction of  trivial principal $\cO(SU_q(2))$-comodule algebras.

\subsection*{Notation}
All algebras in this paper are over the field of complex numbers. They are associative, unital and  $*$-algebras (the latter minor assumption might be dropped and then the choice of the ground field may be freed). Hopf algebras are assumed to have bijective antipodes. 
The comultiplication in a Hopf algebra $H$ is denoted by $\Delta$, counit by $\varepsilon$ and the antipode by $S$. A subscript is used sometimes if more than one Hopf algebra appears. We use the standard Sweedler notation for comultiplication, $\Delta(h)=h\sw{1}\otimes h\sw{2}$ (summation implicit), for all $h\in H$. Whenever needed the coaction on a right $H$-comodule $M$ is denoted by $\rho^H$, and the Sweedler notation $\rho^H(m) = m\sw 0\otimes m\sw 1$ is used. The vector space of right $H$-colinear maps from $M$ to $N$ is denoted by $\Hom^H(M,N)$. By a (right) $H$-comodule algebra we mean an algebra and a right $H$-comodule with a coaction that is an algebra map. 
By $\sigma$ we denote the flip map between vector spaces, $\sigma: V\otimes W \to W\otimes V$, $v\otimes w \mapsto w\otimes v$.

\section{Principal comodule algebras}\label{sec.pri}
\setcounter{equation}{0}

Principal comodule algebras are simply the same as faithfully flat Hopf-Galois extensions (by a Hopf algebra with bijective antipode). The definition of a principal comodule algebra can equivalently be formulated in terms of {\em strong connections}. 

\begin{definition}[see e.g.\ \cite{bh04}, \cite{hkmz}]
Let $H$ be a Hopf algebra with bijective antipode, and let $A$ be a right $H$-comodule algebra with coaction $\rho^H$, multiplication  $\mu:A\otimes A\to A$ and unit $\eta:\C\to A$.
A map 
$$\ell:H\longrightarrow A\otimes A$$
is called a {\em strong connection} if
\begin{subequations}
\label{strong}
\begin{gather}
\ell(1)=1\otimes 1, \label{strong1}\\
\mu\circ \ell = \eta \circ \varepsilon, \label{strong2}\\
 (\ell\otimes\id)\circ\Delta  = (\id\otimes \rho^H)\circ \ell , \label{strong3}\\
(S\otimes \ell)\circ\Delta = (\sigma\otimes \id)\circ (\rho^H\otimes \id)\circ \ell .  \label{strong4}
\end{gather}
\end{subequations}
If an $H$-comodule algebra $A$ admits a strong connection, then it is called a {\em principal $H$-comodule algebra}.
\end{definition}

\begin{remark}\label{rem.trans}
Let us comment how the existence of a strong connection is related to the Hopf-Galois condition. 
Denote by 
$$
B=A^{\co H}:= \{ b\in A\; |\; \rho^H(b) = b\otimes 1\}
$$ 
the subalgebra of $H$-coaction invariants or $H$-coinvariants.
Consider the map 
$$
\can:A\otimes_B A\rightarrow A\otimes H,\quad p\otimes q\longmapsto pq\sw{0}\otimes q\sw{1},
$$
called the {\em canonical map}. $A$ is called a {\em Hopf-Galois extension} of $B$, provided the canonical map is bijective. Denote by $\psi:A\otimes A\rightarrow A\otimes_BA$ the natural surjection.
If $A$ is a principal extension, then the map $\can$ is invertible, and the inverse can be written explicitly as
$$
\can^{-1}:p\otimes h\longmapsto p\psi(\ell(h)).
$$
Hence $\psi(\ell(h))=\can^{-1}(1\otimes h)$. While the strong connection is non-unique, its projection
on the tensor product over $B$, $\psi\circ\ell$ (called the {\em  translation map}), is. \end{remark}

Geometrically, one should understand principal comodule algebras as coordinate algebras of  quantum principal bundles. From this point of view $H$ is the algebra of functions on the fibre (structure quantum group) and the algebra of functions on the base is identified with the coaction invariant subalgebra. This intuitive point of view can also be argued categorically using synthetic approach to noncommutative geometry \cite{Brz:syn}.

{\em Cleft extensions} are examples of principal comodule algebras. These are principal comodule algebras for which there exists a  right $H$-colinear map 
$j: H\to A$ that is unital, i.e.\ $j(1)=1$ and convolution invertible, i.e.\ there exists a map $j^{-1}: H\to A$ such that, for all $h\in H$,
$$
j(h\sw 1)j^{-1}(h\sw 2) =  j^{-1}(h\sw 1)j (h\sw 2) = \varepsilon(h)1.
$$
The map $j$ is referred to as a {\em cleaving map}. In this case a strong connection can be defined as the composite 
\begin{equation}\label{lj}
\ell = (j^{-1}\otimes j)\circ \Delta. 
\end{equation}
Equivalently, cleft extensions can be characterised as those principal $H$-comodule algebras that are isomorphic to $A^{coH}\otimes H$ as left $A^{coH}$-modules and right $H$-comodules (the right $H$-coaction on $A^{coH}\otimes H$ is $\id\otimes \Delta)$. The isomorphism $\theta : A \to A^{coH}\otimes H$, $a\mapsto a\sw 0 j^{-1}(a\sw 1)\otimes h$ induces a twisted tensor product or a crossed product with invertible cocycle algebra structure on $A^{coH}\otimes H$ \cite{DoiTak:cle}. 

Let $A$ be a right $H$-comodule algebra. If there exists a right $H$-colinear algebra map $j: H\to A$, then $j$ is convolution invertible with $j^{-1} = j\circ S$, so $A$ is a cleft principal comodule algebra. In this case, the algebra $A$ is isomorphic to  the smash product of $A^{coH}$ with $H$ (a crossed product with a trivial cocycle) \cite{DoiTak:equ}, and it has a geometric meaning of a trivial quantum principal bundle. Thus we refer to such principal comodule algebras as  {\em trivial principal comodule algebras}.

The main construction used in this paper is given in the following

\begin{lemma}\label{strcot} 
Let $\bar{A}$ be a principal $\bar{H}$-comodule algebra, with a strong connection
$\ell$. Denote 
by $B=\bar{A}^{\co \bar{H}}$  a subalgebra of
$\bar{H}$-coaction invariant elements. Let $\pi:H\rightarrow \bar{H}$ be a Hopf algebra map and consider the cotensor product
$$
\bar{A}\square_{\bar{H}}H := \{\sum_i a^i\otimes h^i \in A\ot H \; |\; \sum_i a^i\sw 0 \otimes a^i\sw 1 \otimes h^i = \sum_i a^i \otimes \pi(h^i\sw 1) \ot h^i\sw 2\}.
$$
View $\bar{A}\square_{\bar{H}}H$ as a right $H$-comodule subalgebra of the tensor algebra $A\otimes H$ with the coaction $\id\otimes \Delta_H$. Then  $\bar{A}\square_{\bar{H}}H$ is a principal $H$-comodule algebra, with
a strong connection 
defined as the composite
$$
\left( \sigma\otimes\id \otimes\id\right)\circ \left(S\otimes \left(\ell\circ\pi\right)\otimes \id\right)\circ (\id\otimes \Delta_H)\circ\Delta_H: H \longrightarrow (\bar{A}\square_{\bar{H}}H)\otimes(\bar{A}\square_{\bar{H}}H).
$$
\end{lemma}

The principal $H$-comodule algebra $A$ constructed in Lemma~\ref{strcot} is known as a {\em prolongation} of $\bar{A}$. Obviously, one can talk of prolongations also in the case of comodule algebras (not necessarily principal). There is no guarantee, however, that the prolonged $H$-comodule algebra $\bar{A}\square_{\bar{H}}H$ be principal. One particular case in which $\bar{A}\square_{\bar{H}}H$ is principal even though $\bar{A}$ is not necessarily so is discussed in Section~\ref{sec.cleft}. In general, by the standard coalgebra-theoretic arguments (the Hom-cotensor relations),
\begin{equation}\label{coinv.prol}
({\bar{A}\square_{\bar{H}}H})^{co {H}} \simeq \Hom^{{H}}(\C, \bar{A}\square_{\bar{H}}H) \simeq \Hom^{\bar{H}}(\C, \bar{A}) \simeq \bar{A}^{co \bar{H}} ,
\end{equation}
so the coaction-invariant subalgebras of both original and prolonged comodule algebras are isomorphic to each other.

The process opposite to prolongation is known as  {\em reduction}. The main result in this area is the Hopf-Galois Reduction Theorem; see \cite{g-r99}, \cite{s-p99} and \cite{qsng}. In this text 
we use one lemma, which can be viewed as a particular corollary of the Hopf-Galois Reduction Theorem.

\begin{lemma}
\label{hogare}
Let $\pi:H\rightarrow \bar{H}$ be a surjective Hopf algebra map (so that $H$ is a left $\bar{H}$-comodule algebra with coaction $(\pi\otimes \id)\circ \Delta_H$) such that $H$ is a left principal $\bar{H}$-comodule algebra. 
Let $\bar{A}$ be a right $\bar{H}$-comodule algebra.  
If $\bar{A}\square_{\bar{H}}H$ is a principal $H$-comodule  algebra, then $\bar{A}$ is a principal $\bar{H}$-comodule algebra.
 \end{lemma}

\section{Quantum spheres and real projective spaces}\label{sec.sph.proj}
\setcounter{equation}{0}
The noncommutative or quantum (Euclidean) spheres were introduced in \cite{P-P87} (in dimension 2) and in \cite{VakSob:alg}, \cite{FRT} (for all $n$). Let  $q$ be a real number, $0< q<1$. 
The coordinate algebra $\cO(S^{2n+1}_q)$ of  the odd-dimensional quantum sphere is the unital complex $*$-algebra with generators $z_0,z_1,\ldots ,z_n$ subject to the following relations:

\begin{subequations}
\label{sph}
\begin{gather}
\label{sph1}
z_iz_j = qz_jz_i \quad \mbox{for $i<j$}, \qquad z_iz^*_j = qz_j^*z_i \quad \mbox{for $i\neq j$},\\
\label{sph2}
z_iz_i^* = z_i^*z_i + (q^{-2}-1)\sum_{m=i+1}^n z_mz_m^*, \qquad \sum_{m=0}^n z_mz_m^*=1.
\end{gather}
\end{subequations}
The coordinate algebra $\cO(S^{2n}_q)$ of  the even-dimensional quantum sphere is the unital complex $*$-algebra with generators $z_0,z_1,\ldots ,z_n$ and relations \eqref{sph} supplemented with $z_n^*=z_n$.\footnote{We have learnt of the possibility of presenting even and odd dimensional quantum spheres is a uniform way from \cite{HawLan:fred}.}  All these quantum spheres are right comodule algebras over the Hopf algebra $\mathcal{O}(\mZt)$ generated by a self-adjoint  grouplike element $u$ satisfying $u^2=1$ (thus $u$ is also unitary). The coaction is defined on generators by
\begin{equation}\label{z2coact}
z_i \longmapsto z_i\otimes u.
\end{equation}
The coordinate algebra of the quantum real projective space
$\cO(\rp_q^m)$ is defined as the $\mathcal{O}(\mZt)$-coaction invariant subalgebra of 
$\cO(S^{m}_q)$.\footnote{In \cite{HonSzy:sph}  the $C^*$-algebras of continuous functions on $\rp_q^m$  were defined and shown to be isomorphic to Cuntz-Krieger algebras associated to suitable directed graphs. By this means Hong and Szyma\'nski were also able to calculate the $K$-theory of quantum real projective spaces.} 
Thanks to the second of relations \eqref{sph2} (the radius relation), every quantum sphere  $\cO(S^{2n}_q)$, $\cO(S^{2n+1}_q)$   admits a strong connection
\begin{equation}\label{strong.sph}
\ell(u) = \sum_{i=0}^n z_i \otimes z_i^*.
\end{equation}
Thus each of the quantum spheres  $\cO(S^{m}_q)$ is a principal $\mathcal{O}(\mZt)$-comodule algebra or a quantum principal bundle over the quantum real projective space $\rp_q^m$. In the case $m=2$ this  was  proven in \cite{h-pm96}, where also the algebra $\cO(\rp_q^2)$ was defined.

Quantum spheres form a hierarchy of right $\mathcal{O}(\mZt)$-comodule $*$-algebras
\begin{equation}\label{hier}
\xymatrix{
\cdots \ar[r]^-{f_5} & \cO(S^{5}_q) \ar[r]^-{f_4} & \cO(S^{4}_q) \ar[r]^-{f_{3}} & \cO(S^{3}_q) \ar[r]^-{f_{2}}&  \cO(S^{2}_q) \ar[r]^-{f_{1}} & \cO(S_q^{1}),}
\end{equation}
where each of the $f_m$ is a surjective $*$-algebra and  right $\mathcal{O}(\mZt)$-colinear map defined on generators as follows. In the odd case
\begin{equation}\label{odd}
f_{2n-1}: \cO(S^{2n}_q)\longrightarrow \cO(S^{2n-1}_q), \qquad z_i\longmapsto  \left\{ \begin{array}{ll}
 z_i & \mbox{if $i\neq n$} \\
0 & \mbox{if $i= n$} 
\end{array}
\right.
\end{equation}
In the even case
\begin{equation}\label{even}
f_{2n}: \cO(S^{2n+1}_q)\longrightarrow \cO(S^{2n}_q), \qquad z_i\longmapsto z_i.
\end{equation}
Note that in this case both $z_n$ and $z_n^*$ are mapped to the same self-adjoint element, which is consistent with the algebraic relations, since $n$ is the maximal number in the set indexing generators of  $ \cO(S^{2n+1}_q)$ (so that \eqref{sph1} are preserved) and $z_n$ is normal by \eqref{sph2}.

The last three members of hierarchy \eqref{hier} are of particular interest, so some comments on them are now in order. In the lowest dimensional case, the relations \eqref{sph} do not depend on the parameter $q$. The algebra $\cO(S_q^{1})$ is a commutative polynomial algebra generated by a unitary element, say $v$, hence it can be identified with the algebra of polynomials on the circle, $\cO(S^{1})$, or the algebra of polynomials on the group $U(1)$, $\cO(U(1))$. $\cO(S^{1})$ is a Hopf algebra, $v$ is a grouplike element and there is an (obvious) Hopf algebra map 
\begin{equation}\label{pi0}
\pi_2: \cO(U(1))=\cO(S^{1})\longrightarrow \mathcal{O}(\mZt), \qquad v\longmapsto u.
\end{equation}
 From now on we use $u$ to denote the unitary (and self-adjoint) generator of $\mathcal{O}(\mZt)$ and we use $v$ to denote the unitary generator of $\cO(S^{1})$.

The one before the penultimate member of the hierarchy, $\cO(S_q^{3})$, is the coordinate algebra of  the quantum group $SU_q(2)$; see \cite{w-sl87}. In terms of generators $a:=z_0$ and $b:= z_1^*$ the relations \eqref{sph} come out as:
$$
ab=qba,\quad ab^*=qb^*a,\quad bb^*=b^*b,\quad aa^*+bb^*=1,\quad a^*a+q^{-2}bb^*=1.
$$
The Hopf algebra structure is given by the matrix co-representation:
\begin{gather*}
\Delta:\left(\begin{matrix}
a & b\\
-q^{-1}b^* & a^*
\end{matrix}\right)\longmapsto \left(\begin{matrix}
a\otimes 1 & b\otimes 1\\
-q^{-1}b^*\otimes 1 & a^*\otimes 1
\end{matrix}\right) \left(\begin{matrix}
1\otimes a & 1\otimes b\\
1\otimes -q^{-1}b^* & 1\otimes a^*
\end{matrix}\right),\nonumber\\
S:\left(\begin{matrix}
a & b\\
-q^{-1}b^* & a^*
\end{matrix}\right)\longmapsto
\left(\begin{matrix}
a^* & -q^{-1}b\\
b^* & a
\end{matrix}\right),\quad
\varepsilon:\left(\begin{matrix}
a & b\\
-q^{-1}b^* & a^*
\end{matrix}\right)\longmapsto
\left(\begin{matrix}
1 & 0\\
0 & 1
\end{matrix}\right).
\end{gather*}
The composite $f_1\circ f_2: \cO(SU_q(2)) \to \cO(U(1))$ is given on generators by $a\mapsto v$, $b\mapsto 0$ and turns out to be a Hopf algebra map. Consequently, the composite 
\begin{equation}\label{pi} 
\pi = \pi_2\circ f_1\circ f_2 : \cO(SU_q(2)) \longrightarrow \mathcal{O}(\mZt), \qquad a\longmapsto u, \quad b\longmapsto 0,
\end{equation}
 is a Hopf algebra map. Since $\pi$ is a Hopf algebra map it makes $\cO(SU_q(2))$ into a (cocentral) $\mathcal{O}(\mZt)$-bicomodule by composing $\pi$ with the comultiplication. The right coaction coincides with the coaction \eqref{z2coact}. The map
\begin{equation}\label{iota}
\imath : \mathcal{O}(\mZt) \longrightarrow  \cO(SU_q(2)), \qquad 1\longmapsto 1, \quad u\longmapsto a,
\end{equation}
is a bicolinear splitting of $\pi$.

The penultimate member of the hierarchy $\cO(S^{2}_q)$ is the coordinate algebra of  the quantum equatorial Podle\'s sphere $S^2_{\sqrt{q},\infty}$  (note the square root in the parameter!) \cite{P-P87}. The map  $i: \cO(S^{2}_q)\to \cO(SU_{\sqrt{q}}(2))$,  $z_0\mapsto -q^{-1}ab$, $z_1\mapsto bb^*$,  is a $*$-algebra inclusion. Furthermore, $i( \cO(S^{2}_q))$ is a left coideal of the Hopf algebra  $\cO(SU_{\sqrt{q}}(2))$, i.e.\ the quantum 2-sphere $S^2_q$ is a quantum homogeneous space of the quantum group $SU_{\sqrt{q}}(2)$.

The algebra of functions on the quantum projective space $\cO(\rp_q^2)$ is generated by 
$$
P=q^{-2} z_1^2,\quad R=z_0^2,\quad T=q^{-1}z_1z_0.
$$
Generators $P$ and $R$ satisfy the following relations \cite{hms03p}:
\begin{gather*}
P=P^*,\quad T^2=qPR,\quad RT^*=qT(-q^2P+1),\quad R^*T=q^{-1}T^*(-P+1),\nonumber\\
RR^*=q^{6}P^2-q^2(1+q^2)P+1,\quad R^*R=q^{-2}P^2-(1+q^{-2})P+1,\nonumber\\
TT^*=-q^2P^2+P,\quad T^*T=q^{-2}(P-P^2),\nonumber\\
RP=q^4PR,\quad RT=q^2TR,\quad PT=q^{-2}TP.
\end{gather*}
Equivalently, the coordinate algebra of the quantum projective space $\cO(\rp_q^2)$ is a $*$-algebra generated by $P$, $R$ and $T$ satisfying  above relations  (note that the squaring of $q$ is needed to  synchronise our conventions with those of \cite{hms03p}).

\section{Cleft prolongations of comodule algebras}\label{sec.cleft}
\setcounter{equation}{0}
The aim of this section is to determine necessary and sufficient conditions for a prolongation of a right comodule algebra to be cleft. Throughout this section $\pi:H\rightarrow \bar{H}$ is a Hopf algebra map and $H$ is understood as an $\bar{H}$-bicomodule with coactions $(\pi\otimes \id)\circ \Delta_H$, $(\id\otimes \pi)\circ \Delta_H$. Furthermore, $\bar{A}$ is a right $\bar{H}$-comodule algebra and 
 $B=\bar{A}^{\co\bar{H}}$. 

\begin{proposition}\label{cleftprop}
The prolongation $\bar{A}\square_{\bar{H}}H$ is a cleft extension of $B$ if and only if there
exists a right $\bar{H}$-colinear, unital, convolution invertible map $f:H\rightarrow \bar{A}$.
\end{proposition}
\begin{proof}
Given $f$ define $\theta:\bar{A}\square_{\bar{H}}H\rightarrow B\otimes H$ by
$\theta(a\otimes h)=af^{-1}(h\sw{1})\otimes h\sw{2}$. Obviously $\theta$ is a left $B$-module map and a right $H$-comodule map.
The convolution inverse $f^{-1}$ satisfies the 
following covariance property:
\begin{equation}\label{invcov}
f^{-1}(h\sw{2})\otimes S\left(\pi(h\sw{1})\right)=f^{-1}(h)\sw{0}\otimes f^{-1}(h)\sw{1}.
\end{equation}
This can be easily proven by standard Hopf algebraic techniques as follows.
Consider the identity:
\begin{equation*}
f^{-1}(h\sw{1})f(h\sw{2})\otimes S\left(\pi(h\sw{3})\right)\otimes f^{-1}(h\sw{4})=1\otimes S\left(\pi(h\sw{1})\right)\otimes f^{-1}(h\sw{2}).
\end{equation*}
Apply the $\bar{H}$-coaction to the first tensorand and use the facts that the coaction is an algebra map and that $f$ is right $\bar{H}$-colinear to obtain:
\begin{multline*}
f^{-1}(h\sw{1})\sw{0}f(h\sw{2})\otimes f^{-1}(h\sw{1})\sw{1}\pi(h\sw{3}) \otimes  S\left(\pi(h\sw{4})\right)\otimes f^{-1}(h\sw{5})\\
=1\otimes 1\otimes S\left(\pi(h\sw{1})\right)\otimes f^{-1}(h\sw{2}).
\end{multline*}
Multiply the middle legs and use the assumption that $\pi$ is a Hopf algebra map to derive the following equality:
\begin{equation*}
f^{-1}(h\sw{1})\sw{0}f(h\sw{2})\otimes f^{-1}(h\sw{1})\sw{1}\otimes f^{-1}(h\sw{3})
=1\otimes S\left(\pi(h\sw{1})\right)\otimes f^{-1}(h\sw{2}).
\end{equation*}
Finally multiplication of the first and the third tensorands yields the desired formula.

We use \eqref{invcov} to prove that the image of $\theta$ is in $B\otimes H$. Apply $\rho^{\bar{H}}\otimes H$ to $\theta(a\otimes h)$,
where $a\otimes h\in \bar{A}\square_{\bar{H}}H$ (implicit supressed summation). Then, 
\begin{multline*}
\rho^{\bar{H}}(af^{-1}(h\sw{1}))\otimes h\sw{2}=
a\sw{0}f^{-1}(h\sw{1})\sw{0}\otimes a\sw{1}f^{-1}(h\sw{1})\sw{1}\otimes h\sw{2}\\
=a\sw{0}f^{-1}(h\sw{2})\otimes a\sw{1}S\left(\pi(h\sw{1})\right)\otimes h\sw{3}
=af^{-1}(h\sw{3})\otimes \pi(h\sw{1})S\left(\pi(h\sw{2})\right)\otimes h\sw{4}\\
=af^{-1}(h\sw{1})\otimes 1\otimes h\sw{2},
\end{multline*} 
where the fact that  $a\otimes h\in \bar{A}\square_{\bar{H}}H$ was used in the derivation of the third equality.

The inverse of $\theta$ is:
\begin{equation*}
\theta^{-1}:B\otimes H\ni b\otimes h\longmapsto bf(h\sw{1})\otimes h\sw{2}\in \bar{A}\square_{\bar{H}}H.
\end{equation*}
The right $H$-colinearity of $f$ ensures that the map $\theta^{-1}$ has the specified codomain. A straightforward calculation that employs convolution-invertibility of $f$ confirms that $\theta^{-1}$ is really the inverse of $\theta$ as claimed.
Hence $B\otimes H$ is a cleft extension.

Conversely, suppose that $\bar{A}\square_{\bar{H}}H$ is a cleft extension. Then there exists a convolution invertible, normalized
map $j\in \Hom^H(H,\bar{A}\square_{\bar{H}}H)$. In view of the identification
$\Hom^H(H,\bar{A}\square_{\bar{H}}H)\simeq \Hom^{\bar{H}}(H,\bar{A})$ the existence of a right $H$-colinear map $j$ is equivalent to the existence of the right $\bar{H}$-colinear map
 $f:=(\id\square_{\bar{H}}\varepsilon_H)\circ j$. Write $j(h)=h^{(1)}\otimes h^{(2)}$ so that
$f(h)=h^{(1)}\varepsilon_H(h^{(2)})$. Since $j(1_H)=1_{\bar A}\otimes 1_H$,  $f(1_H)=1_{\bar A}$. Write 
$j^{-1}(h)=h^{[1]}\otimes h^{[2]}$  for the convolution inverse of $j$. Then 
$f^{-1}(h):=h^{[1]}\varepsilon_H(h^{[2]})$ is the convolution inverse of $f$, since
\begin{multline*}
f(h\sw{1})f^{-1}(h\sw{2})=h\sw{1}{}^{(1)}\varepsilon_H(h\sw{1}{}^{(2)})
h\sw{2}{}^{[1]}\varepsilon_H(h\sw{2}{}^{[2]})
=h\sw{1}{}^{(1)}h\sw{2}{}^{[1]}\varepsilon_H(h\sw{1}{}^{(2)}h\sw{2}{}^{[2]})\\
=\varepsilon_H(h)\varepsilon_H(1)=\varepsilon_H(h),
\end{multline*}
as required.
\end{proof}

In the setup of Proposition~\ref{cleftprop},
the algebra structure on $B\otimes H$ induced by $\theta$ is:
\begin{multline}\label{crossed}
(b\otimes h)(c\otimes g)
=\theta(\theta^{-1}(b\otimes h)\theta^{-1}(c\otimes g))
=\theta((bf(h\sw{1})\otimes h\sw{2})(cf(g\sw{1})\otimes g\sw{2}))\\
=\theta(bf(h\sw{1})cf(g\sw{1})\otimes h\sw{2}g\sw{2})
=bf(h\sw{1})cf(g\sw{1})f^{-1}(h\sw{2}g\sw{2})\otimes h\sw{3}g\sw{3}, 
\end{multline} 
for all $b,c\in B$ and $g,h\in H$. We write $B\# H$ for the vector space $B\otimes H$ with this algebra structure. 

Proposition~\ref{cleftprop} implies that if there exists a unital, convolution invertible and right $\bar{H}$-colinear map $f$, then the prolongation $\bar{A}\square_{\bar{H}}H$ is a principal $H$-comodule algebra. In terms of the map $f$, the  standard strong connection for a cleft comodule algebra \eqref{lj} comes out as
\begin{equation}\label{f.strong}
\ell(h) = f^{-1}(h\sw 2)\otimes Sh\sw 1 \otimes f(h\sw 3)\otimes h\sw 4 \in \bar{A}\square_{\bar{H}}H\otimes \bar{A}\square_{\bar{H}}H,
\end{equation}
for all $h\in H$.
 If no further assumptions on $\pi$ are made, we cannot yet conclude that $\bar{A}$ is a principal $\bar{H}$-comodule algebra.  More can be said if $H$ is a left principal $\bar{H}$-comodule algebra.

\begin{corollary}\label{cor.princ}
Assume that $\pi$ is onto and that  $H$ is a left principal $\bar{H}$-comodule algebra. If there
exists a right $\bar{H}$-colinear, unital, convolution invertible map $f: H\rightarrow \bar{A}$, then $\bar{A}$ is a principal right $\bar{H}$-comodule algebra.
\end{corollary}
\begin{proof} By Proposition~\ref{cleftprop}, $\bar{A}\square_{\bar{H}}H$ is a principal $H$-comodule algebra,  hence $\bar{A}$ is a principal $\bar{H}$-comodule algebra by 
 Lemma~\ref{hogare}. 
\end{proof}

In some situations a strong connection in $\bar{A}$ can be explicitly written down. The following corollary discusses the case applicable to quantum spheres and real projective spaces recalled in Section~\ref{sec.sph.proj}.

\begin{corollary}
\label{strpropp}
Assume that there 
exists a right $\bar{H}$-colinear, unital, convolution invertible map $f:H\rightarrow \bar{A}$. If there exists an $\bar{H}$ bi-colinear section $\imath:\bar{H}\rightarrow H$ of $\pi$ such that
$\imath(1_{\bar{H}})=1_{{H}}$,  then
 \begin{equation}
\ell:\bar{H}\longrightarrow \bar{A}\otimes \bar{A},\quad 
h\longmapsto f^{-1}(\imath(h)\sw{1})\otimes f(\imath(h)\sw{2}), 
\end{equation}
is a strong connection. Consequently,
$\bar{A}$ is a principal $\bar{H}$-comodule algebra.
\end{corollary}
\begin{proof}
Since in this case $H$ is a principal left $\bar{H}$-comodule algebra \cite[Corollary~2.6]{HajMaj:pro}, the existence of a strong connection is indeed a corollary of Proposition~\ref{cleftprop} and Lemma~\ref{hogare}. 
Properties \eqref{strong} can be  checked directly as follows. 
By definition, for any $h\in\bar{H}$,  $\pi(\imath(h))=h$, therefore
$\varepsilon_{\bar{H}}(\pi(\imath(h)))=\varepsilon_{\bar{H}}(h)$. Then, 
\begin{equation}\label{imatheps}
\varepsilon_{H}(\imath(h))=\varepsilon_{\bar{H}}(h),
\end{equation}
since $\pi$ is a Hopf algebra map.
It follows that, for all $h\in\bar{H}$,
$$
\ell(1)=f^{-1}(\imath(1_{\bar{H}})\sw{1})\otimes f(\imath(1_{\bar{H}})\sw{2})
=f^{-1}(1_{H}\sw{1})\otimes f(1_{H}\sw{2})=1_{\bar{A}}\otimes 1_{\bar{A}},
$$
where we used the unitality of $f$, $f^{-1}$ and $\imath$. Using eq.~\eqref{imatheps} we obtain
$$
f^{-1}(\imath(h)\sw{1})f(\imath(h)\sw{2})=\varepsilon_H(\imath(h))=\varepsilon_{\bar{H}}(h),
$$
i.e.\ \eqref{strong2} holds. To prove the right covariance \eqref{strong3} of $\ell$ we use the right $H$-colinearity of $f$ and the right $\bar{H}$-colinearity
of $\imath$:
\begin{multline*}
\ell(h\sw{1})\otimes h\sw{2}=f^{-1}(\imath(h\sw{1})\sw{1})\otimes f(\imath(h\sw{1})\sw{2})\otimes h\sw{2}\\
=f^{-1}(\imath(h)\sw{1})\otimes f(\imath(h)\sw{2})\otimes\pi(\imath(h)\sw{3})
=f^{-1}(\imath(h)\sw{1})\otimes f(\imath(h)\sw{2})\sw{0}\otimes f(\imath(h)\sw{2})\sw{1}.
\end{multline*}
Finally, using eq.~\eqref{invcov} we obtain
\begin{multline*}
f^{-1}(\imath(h)\sw{1})\sw{1}\otimes f^{-1}(\imath(h)\sw{1})\sw{0}\otimes f(\imath(h)\sw{2})\\
=S(\pi ( \imath(h) \sw{1} ))\otimes f^{-1}(\imath(h)\sw{2})\otimes f(\imath(h)\sw{3})
=S(h\sw{1}) \otimes f^{-1}(\imath(h\sw{2})\sw 1)\otimes f(\imath(h\sw{2})\sw 2),
\end{multline*}
where, in the derivation of the last equality, we used the left colinearity of $\imath$. This proves the left $\bar{H}$-colinearity of $\ell$, eq.~\eqref{strong4}.
\end{proof}

Since a principal $H$-comodule algebra $A$ is in particular a Hopf-Galois extension of coinvariants $B$ (see Remark~\ref{rem.trans}), the Hopf algebra $H$ acts from the right on the centraliser subalgebra and $H$-subcomodule  of $A$,
$Z_A(B):=\{a\in A\;|\; ab=ba,\;\forall\; b\in B\}$, by
the
{\em Miyashita-Ulbrich action} \cite{Ulb:Gal}. 
Write $h\o \otimes_B h\t$ for the image  of $h\in H$ under the translation map  (see Remark~\ref{rem.trans}). Then, for all $a\in  Z_A(B)$, the Miyashita-Ulbrich action   is given by
$$
a\triangleleft h := h\o a h\t .
$$
On the other hand,  ${}^{co\bar{H}}H$ is a right $H$-coideal and a right $H$-module by the adjoint coaction,
$$
k\triangleleft h := S(h\sw1)kh\sw2, \quad \mbox{for all}\quad  k\in{}^{co\bar{H}}H,\; h\in H.
$$

The following proposition is
related to  the (full) Hopf-Galois Reduction Theorem \cite{g-r99,s-p99,qsng}
 if $H$ is a left principal $\bar{H}$-comodule algebra.
Namely, the map $\Psi$ gives (in accordance with the Hopf-Galois Reduction Theorem) the reduction of 
$\bar{A}\square_{\bar{H}}H$ to $\bar{A}$.

\begin{proposition}
If there
exists a right $\bar{H}$-colinear, unital, convolution invertible map $f: H\rightarrow \bar{A}$, then
the following map:
\begin{equation}\label{F}
\Psi:{}^{\co\bar{H}}H\longrightarrow Z_{B\#H}(B),\quad k\longmapsto f^{-1}(k\sw{1})\otimes k\sw{2},
\end{equation}
is an algebra, right $H$-colinear and right $H$-linear map. 
\end{proposition}
\begin{proof}
Applying $\rho^{\bar H}\otimes \id$ to $f^{-1}(k\sw{1})\otimes k\sw{2}$ and using \eqref{invcov} and that $k\in  {}^{\co\bar{H}}H$ one immediately finds that $\Psi(k) \in B\otimes H$. The product in the algebra $B\#H$ is given by \eqref{crossed}, hence, for all $b\in B$, $k\in  {}^{\co\bar{H}}H$,
$$
(f^{-1}(k\sw{1})\otimes k\sw{2})(b\otimes 1) =  f^{-1}(k\sw 1)f(k\sw 2)bf^{-1}(k\sw 3)\otimes k\sw{4} = bf^{-1}(k\sw{1})\otimes k\sw{2},
$$
i.e.\ $\Psi(k) \in Z_{B\#H}(B)$. Using \eqref{crossed} one easily finds that $\Psi$ is an algebra map. It is clearly a right $H$-colinear map. To prove the right $H$-linearity of $\Psi$, first note that, for all $k\in  {}^{\co\bar{H}}H$,
$$
\theta^{-1} \circ \Psi (k) = 1\otimes k,
$$
where $\theta^{-1}:  B\#H\to \bar{A}\square_{\bar{H}}H$ is the isomorphism constructed in the proof of Proposition~\ref{cleftprop}. The Miyashita-Ulbrich action can be calculated with the help of the strong connection \eqref{f.strong}
\begin{eqnarray*}
(1\otimes k)\triangleleft h &=&  (f^{-1}(h\sw 2)\otimes Sh\sw 1) (1\otimes k) (f(h\sw 3)\otimes h\sw 4)\\
& = & f^{-1}(h\sw 2)f(h\sw 3)\otimes Sh\sw 1k h\sw 4 
=   1 \otimes k\triangleleft h = \theta^{-1} \circ \Psi (k\triangleleft h) .
\end{eqnarray*}
Therefore, $\Psi$ is a right $H$-module map as required.
\end{proof}

\section{Prolongations of quantum spheres (over quantum real projective spaces)}\label{sec.exam}
\setcounter{equation}{0}

\subsection{Prolongations to $SU_q(2)$-bundles}
As the first illustration of Proposition~\ref{cleftprop} we consider the prolongation of  principal ${\mathcal{O}(\mZt)}$-comodule algebras  $\cO(S^m_q)$ to  principal $\cO(SU_q(2))$-comodule algebras along the Hopf algebra map $\pi: \mathcal{O}(SU_{q}(2))\to \mathcal{O}(\mZt)$ \eqref{pi}. As explained in Section~\ref{sec.sph.proj}, the coordinate algebras of quantum spheres form a hierarchy. Since the algebra maps in \eqref{hier} are also right ${\mathcal{O}(\mZt)}$-colinear, the prolongation gives rise to the following hierarchy of principal  $\mathcal{O}(SU_{q}(2))$-comodule algebras:
\begin{equation}\label{hier2}
\xymatrix{
 \cO(S^{m+1}_q)\square_{\mathcal{O}(\mZt)}\cO(SU_q(2)) \ar[rr]^-{f_{m}\otimes\id} && \cO(S^{m}_q)\square_{\mathcal{O}(\mZt)}\cO(SU_q(2)) }, \quad m=2,3,\ldots
\end{equation}
 Now, the Hopf algebra map $\pi: \mathcal{O}(SU_{q}(2))\to \mathcal{O}(\mZt)$ has a unital bicolinear splitting $\imath$, given in \eqref{iota}. This, in particular, implies that all the maps in \eqref{hier2} are surjective ($\mathcal{O}(SU_{q}(2))$ is coflat as a left $\mathcal{O}(\mZt)$-comodule). Furthermore we can invoke the reduction theorem to conclude that each of the  principal $\mathcal{O}(\mZt)$-comodule algebras  $\mathcal{O}(S^m_{q})$ is a reduction of the principal $\mathcal{O}(SU_{q}(2))$-comodule algebra $\cO(S^{m}_q)\square_{\mathcal{O}(\mZt)}\cO(SU_q(2))$. The coaction invariants of both $\cO(S^{m}_q)\square_{\mathcal{O}(\mZt)}\cO(SU_q(2))$ and $\mathcal{O}(S^m_{q})$ necessarily coincide with $\cO(\rp^m_q)$ (see \eqref{coinv.prol}).

Since $\cO(S^{3}_q) = \cO(SU_q(2))$ more can be said about the last two algebras in hierarchy \eqref{hier2}. First, $f_2: \cO(SU_q(2)) \to \cO(S^2_q)$ (see \eqref{even})  is a right ${\mathcal{O}(\mZt)}$-colinear algebra map.  As an algebra map,  $f_2$ is   convolution invertible with the (convolution) inverse $f^{-1}_2 = f_2\circ S$, i.e.\
$$
f_2^{-1} : a \longmapsto z_0^*, \qquad b\longmapsto -q^{-1}z_1, \qquad a^*\longmapsto z_0, \qquad b^*\longmapsto -qz_1.
$$
Therefore, $\cO(S^2_q)\square_{\mathcal{O}(\mZt)}\cO(SU_q(2))$ is a cleft (in fact trivial) principal comodule algebra by Proposition~\ref{cleftprop}.

Since $f_2$ is an algebra map, there is a left action of $\mathcal{O}(SU_{q}(2))$ on $\mathcal{O}(\rp^2_q)$,
\begin{equation}
h\triangleright x=f_2(h\sw{1})xf_2(S(h\sw{2})),
\end{equation}
and the algebra structure of $\cO(S^2_q)\square_{\mathcal{O}(\mZt)}\cO(SU_q(2))$ calculated from \eqref{crossed}  is that of the smash product $\mathcal{O}(\rp^2_q)\# \mathcal{O}(SU_{q}(2))$,
$$
(x\otimes h)(y\otimes h') = xh\sw 1\triangleright y\otimes h\sw 2h', \qquad \mbox{for all}\quad  x,y\in \mathcal{O}(\rp^2_q),\; h,h'\in \mathcal{O}(SU_{q}(2)).
$$
Using the explicit form of $f_2$ one easily derives the formulae for the action of generators of $\mathcal{O}(SU_{q}(2))$ on generators of $\mathcal{O}(\rp^2_q)$:
$$
a\triangleright P=q^2P+q^2(1-q^2)P^2, \quad 
b\triangleright P=q(1-q^2)PT, \quad 
a\triangleright R=R+q^4(1-q^2)PR, \quad
$$
$$
b\triangleright R=q(1-q^2)TR,\quad 
a\triangleright T=qT+q^3(1-q^2)PT,\quad
b\triangleright T=(1-q^2)T^2.
$$

Thus we can conclude that  the  {\em non-trivial} principal $\mathcal{O}(\mZt)$-comodule algebra  $\mathcal{O}(S^2_{q})$ is a reduction of the {\em trivial} principal $\mathcal{O}(SU_{q}(2))$-comodule algebra $\mathcal{O}(\rp^2_q)\# \mathcal{O}(SU_{q}(2))$ (both over $\mathcal{O}(\rp^2_q)$).

The penultimate algebra in hierarchy \eqref{hier2}, $\cO(S^3_q)\square_{\mathcal{O}(\mZt)}\mathcal{O}(SU_{q}(2))$, is a trivial principal comodule algebra since the identity map $\mathcal{O}(SU_{q}(2))\to \cO(S^3_q)$ fulfills all the assumptions of Proposition~\ref{cleftprop}.

\subsection{Prolongations to $U(1)$-bundles}
One can use  Hopf algebra map $\pi_2: \cO(U(1)) \to \mathcal{O}(\mZt)$ \eqref{pi0} to construct prolongations of the $\cO(S^m_q)$ to principal $\cO(U(1))$-comodule algebras. The sequence of algebra maps \eqref{hier} now yields a sequence of algebra maps between principal $\cO(U(1))$-comodule algebras
\begin{equation}\label{hier3}
\xymatrix{
 \cO(S^{m+1}_q)\square_{\mathcal{O}(\mZt)}\cO(U(1))  \ar[rr]^-{f_{m}\otimes\id} && \cO(S^{m}_q)\square_{\mathcal{O}(\mZt)}\cO(U(1)) }, \quad m=2,3,\ldots
\end{equation}
Since $\pi_2$ is split by a bicomodule map $u\mapsto v$ and each of the $f_m$ is surjective, also all the maps in \eqref{hier3} are surjective.

\begin{proposition}
\label{prop.nontriv}
For all natural numbers $m>1$, the principal $\cO(U(1))$-comodule algebras  $\cO(S^{m}_q)\square_{\mathcal{O}(\mZt)}\cO(U(1))$ are non-trivial.
\end{proposition}
\begin{proof}
This can be proven by induction. If $m=2,3$, then non-zero multiples of the identity are the only units in $\cO(S^m_q)$ \cite{HajMaj:pro}. Since $\cO(U(1))$ is the algebra of Laurent polynomials in one variable and the generator of $\cO(U(1))$ is a group-like element, any convolution invertible map $f: \cO(U(1))\to \cO(S^m_q)$ must have the form $f(v) = \lambda 1$, $\lambda\in \C^\times$. Such a map cannot be right $\mathcal{O}(\mZt)$-colinear, as the coactions send $v$ to $v\ot u$ and $1$ to $1\ot 1$. Therefore, there are no right $\mathcal{O}(\mZt)$-colinear convolution invertible maps $\cO(U(1))\to \cO(S^m_q)$, and $\cO(S^{m}_q)\square_{\mathcal{O}(\mZt)}\cO(U(1))$ is non-cleft (hence non-trivial) by Proposition~\ref{cleftprop} if $m=2,3$.

Take $m >1$  for which there are no convolution invertible $\mathcal{O}(\mZt)$-colinear maps from $\cO(U(1))$ to $\cO(S^{m}_q)$. Suppose there exists a convolution invertible, $\mathcal{O}(\mZt)$-colinear map
$f: \cO(U(1))\to \cO(S^{m+1}_q)$. Then $f_m\circ f$ would be a convolution invertible, 
$\mathcal{O}(\mZt)$-colinear map
from $\cO(U(1))$ to $\cO(S^{m}_q)$ contradicting the inductive assumption.
\end{proof}

\begin{proposition}\label{prop.present}
\begin{enumerate}
\item For all integers $n\geq1$, $\cO(S^{2n+1}_q)\square_{\mathcal{O}(\mZt)}\cO(U(1))$ is an algebra isomorphic to $\cO(S^{2n+1}_q)\otimes \cO(U(1))$.
\item For all integers $n>1$, $\cO(S^{2n}_q)\square_{\mathcal{O}(\mZt)}\cO(U(1))$ is a subalgebra of $\cO(S^{2n}_q)\otimes \cO(U(1))$ isomorphic to a polynomial $*$-algebra 
$\mathcal{A}^{2n}$
generated by $\zeta_0, \zeta_1,\ldots, \zeta_n$ and a central unitary $\xi$ subject to the following relations
\begin{subequations}
\label{kle}
\begin{gather}
\label{kle1}
\zeta_i\zeta_j = q\zeta_j\zeta_i \quad \mbox{for $i<j$}, \qquad \zeta_i\zeta^*_j = q\zeta_j^*\zeta_i \quad \mbox{for $i\neq j$},\\
\label{kle2}
\zeta_i\zeta_i^* = \zeta_i^*\zeta_i + (q^{-2}-1)\sum_{m=i+1}^n \zeta_m\zeta_m^*, \qquad \sum_{m=0}^n \zeta_m\zeta_m^*=1, \qquad \zeta_n^* = \zeta_n\xi.
\end{gather}
\end{subequations}

\end{enumerate}
\end{proposition}
\begin{proof}
The proof of part (1) is based on the following 
\begin{lemma}\label{lemma.ident}
Let $H$ be a commutative Hopf algebra, let $\pi: H\to \bar{H}$ be a Hopf algebra map, and let $A$ be a right $H$-comodule algebra with coaction $\rho^H$. View $A$ as a right $\bar{H}$-comodule algebra by the coaction $\rho^{\bar{H}} = (\id \otimes \pi)\circ\rho^H$ and $H$ as a left $\bar{H}$-comodule algebra by the coaction $(\pi\otimes \id)\circ \Delta_H$. Then the algebra $A\square_{\bar{H}} H$ is isomorphic with $A\otimes {}^{co\bar{H}}H$.
\end{lemma}
\begin{proof}
For all $a\in A$, write $a\sw 0\otimes a\sw 1$ for $\rho^H(a)$. Then the isomorphism is
$$
\kappa: A\otimes {}^{co\bar{H}}H\longrightarrow A\square_{\bar{H}} H, \qquad a\otimes h \longmapsto a\sw 0\otimes a\sw 1h,
$$
with the inverse
$$
\kappa^{-1}: A\square_{\bar{H}} H\longrightarrow A\otimes {}^{co\bar{H}}H, \qquad \sum_i a^i\otimes h^i \longmapsto \sum_i a^i\sw 0 \otimes S(a^i\sw 1)h^i.
$$
This can be checked by a straightforward calculation. We only mention that the commutativity of $H$ ensures that $\kappa$ is an algebra map, and that, in general, $\kappa$ is not an isomorphism of $H$-comodule algebras (it does not respect the obvious coactions obtained by restrictions of $\id\otimes \Delta_H$).
\end{proof}

Observe that each of the odd-dimensional spheres is a right $\cO(U(1))$-comodule algebra with the coaction $z_i\mapsto z_i\ot v$. The $\mathcal{O}(\mZt)$ coaction is obtained from the $\cO(U(1))$-coaction by applying the Hopf algebra map $\pi_2$ \eqref{pi0}. Since $\cO(U(1))$ is a commutative algebra, Lemma~\ref{lemma.ident} can be applied and we deduce that $\cO(S^{2n+1}_q)\square_{\mathcal{O}(\mZt)}\cO(U(1))$ is an algebra isomorphic to $\cO(S^{2n+1}_q)\otimes {}^{co\mathcal{O}(\mZt)}\cO(U(1))$. Note next that $\cO(U(1)) = \C[v,v^*]$ is the algebra of Laurent polynomials ($v^*$ is the inverse of $v$) and the coaction invariant subalgebra consists of all combinations of monomials of even degree, i.e.\ ${}^{co\mathcal{O}(\mZt)}\cO(U(1)) = \C[v^2,v^{*2}]$. The latter is again the algebra of Laurent polynomials in one variable, hence it is isomorphic with $\cO(U(1))$. This completes the proof of part (1).

(2) The vector space $\cO(S^{2n}_q)\square_{\mathcal{O}(\mZt)}\cO(U(1))$ is spanned by elements $x\ot v^k$, where $x$ is a monomial in $\cO(S^{2n}_q)$ with degree congruent to $k$ modulo 2. All such elements can be generated by multiplying
\begin{equation}\label{zetaz}
\zeta_i = z_i \ot v, \qquad \xi = 1\ot v^{*2},
\end{equation}
and their conjugates. Clearly $1\ot v^{*2}$ is central  and unitary, the relations between the $z_i$ are inherited from relations \eqref{sph}. In this way we obtain all but the last of relations \eqref{kle}. Since $z_n$ is self-adjoint, 
$$
\zeta^*_n  = z_n \ot v^* = \zeta_n \xi. 
$$
This implies that equations \eqref{zetaz} define a surjective $*$-algebra map $\Phi$ from the $*$-algebra $\mathcal{A}^{2n}$ generated by $\zeta_i, \xi$ and relations \eqref{kle} to $\cO(S^{2n}_q)\square_{\mathcal{O}(\mZt)}\cO(U(1))$. Employing Diamond Lemma, one easily finds that a basis for $\mathcal{A}^{2n}$ consists of monomials
$$
\zeta_0^{k}\zeta_1^{k_1}\cdots \zeta_n^{k_n}\zeta_1^{*l_1}\cdots \zeta_{n-1}^{*l_{n-1}}\xi^m, \qquad k,m\in \Z,\ k_i,l_i\in \N,
$$
where, by convention, $\zeta_0^{-|k|}$ denotes $\zeta_0^{*|k|}$. Since all these elements are mapped by $\Phi$ to linearly independent vectors
$$
z_0^{k}z_1^{k_1}\cdots z_n^{k_n}z_1^{*l_1}\cdots z_{n-1}^{*l_{n-1}}\ot v^{\sum_ik_i - \sum_il_i +k-2m}, \qquad k,m\in \Z,\ k_i,l_i\in \N, 
$$
 in $\cO(S^{2n}_q)\square_{\mathcal{O}(\mZt)}\cO(U(1))$, $\Phi$ is an injective $*$-algebra map as required.
\end{proof}

Let us make a few comments on the representation theory of algebras constructed in Proposition~\ref{prop.present}. Due to the close relationship of these algebras with coordinate algebras of quantum spheres, their representation theory bears close resemblance to that of the latter (see e.g.\ \cite[Section~3]{HawLan:fred} for the detailed discussion of representations of quantum spheres). We concentrate on the (more interesting) algebras in part (2) of Proposition~\ref{prop.present}. Non-equivalent irreducible $*$-representations of $\mathcal{A}^{2n}\simeq \cO(S^{2n}_q)\square_{\mathcal{O}(\mZt)}\cO(U(1))$ split into two classes depending on whether $\zeta_n$ is represented by a non-zero or the zero operator. First, there is a family of 
representations labeled by  $\phi\in (0,2\pi)$ and the sign $\pm$.  For each $\phi$ and $\pm$ the representation space $V_{\phi,\pm}$ of $\pi_{\phi,\pm}$ has an orthonormal basis: $|k_0,k_1, \ldots, k_{n-1}\rangle$, $k_i=0,1,2,\ldots$. On this basis of $V_{\phi,\pm}$, the (bounded) operators representing $\zeta_i,\xi$ act as follows: 
 \begin{subequations}
\label{rep.1}
\begin{gather}
 \pi_{\phi,\pm}(\zeta_n) |k_0, \ldots, k_{n-1}\rangle = \pm e^{i\phi} q^{k_0+\ldots +k_{n-1} +n}|k_0, \ldots, k_{n-1}\rangle ,\\
 \pi_{\phi,\pm}(\zeta_l) |k_0, \ldots, k_{n-1}\rangle =(1-q^{2k_l})^{1/2}q^{k_0+\ldots +k_{l-1} +l}|k_0, \ldots, k_l -1,  \ldots k_{n-1}\rangle , \quad l<n, \\
 \pi_{\phi,\pm}(\xi) |k_0, \ldots, k_{n-1}\rangle =  e^{-2i\phi} |k_0, \ldots, k_{n-1}\rangle .
\end{gather}
\end{subequations}
Second, if $\zeta_n$ is represented by the zero operator, then the $S_q^{2n}$-sphere part of the algebra $\cO(S^{2n}_q)\square_{\mathcal{O}(\mZt)}\cO(U(1))$ collapses to the odd-dimensional quantum sphere $S^{2n-1}_q$, and we are essentially in the situation described in part (1) of Proposition~\ref{prop.present}. Therefore,   representations 
of $\mathcal{A}^{2n}$
are given by the tensor product of those for the odd-dimensional quantum sphere $S_q^{2n-1}$ and the circle group $U(1)$. Explicitly, there is a family of representations $\pi_{\lambda,\mu}$ labeled by $\lambda$, $\mu$ such that $ |\lambda|=|\mu|=1$. The orthonormal basis for the corresponding representation space $V_{\lambda,\mu}$ is $|k_0,k_1, \ldots, k_{n-1}\rangle$, $k_i=0,1,2,\ldots$. On this basis of $V_{\lambda,\mu}$, the (bounded) operators representing $\zeta_i,\xi$ act as follows: 
 \begin{subequations}
\label{rep.2}
\begin{gather}
\label{rep.2.1}
 \pi_{\lambda,\mu}(\zeta_n) = 0, \qquad  \pi_{\lambda,\mu}(\zeta_{n-1}) |k_0, \ldots, k_{n-1}\rangle = \lambda q^{k_0+\ldots +k_{n-1} +n}|k_0, \ldots, k_{n-1}\rangle ,\\
 \pi_{\lambda,\mu}(\zeta_l) |k_0, \ldots, k_{n-1}\rangle =(1-q^{2k_l})^{1/2}q^{k_0+\ldots +k_{l-1} +l}|k_0, \ldots, k_l -1,  \ldots k_{n-1}\rangle , \;\;\; l<n-1,\\
 \pi_{\lambda,\mu}(\xi) |k_0, \ldots, k_{n-1}\rangle = \mu |k_0, \ldots, k_{n-1}\rangle .
\end{gather}
\end{subequations}

In the case of algebras $\cO(S^{2n-1}_q)\square_{\mathcal{O}(\mZt)}\cO(U(1))$, by the part (1) of Proposition~\ref{prop.present} irreducible representations have the form \eqref{rep.2} (without the first of equations \eqref{rep.2.1}), where the $\zeta_l$ should be replaced by the $z_l$ and $\xi$ corresponds to the unitary generator of $\cO(U(1))$.

The representations $\pi_{\phi,\pm}$ can be used to construct Fredholm modules \cite[Chapter~4]{Con:non} over the algebras  $\mathcal{A}^{2n}\simeq  
\cO(S^{2n}_q)\square_{\mathcal{O}(\mZt)}\cO(U(1))$ in the same way Fredholm modules over coordinate algebras of even quantum spheres are constructed in \cite{HawLan:fred}. For each $\phi$, an even Fredholm module $(\mathcal{H}_\phi, F,\gamma)$ is given by the representation
$$
\pi_\phi = \pi_{\phi,+} \oplus \pi_{\phi,-} \qquad \mbox{on}\qquad \mathcal{H}_\phi = V_{\phi,+} \oplus V_{\phi,-},
$$
 with operators 
 $$
 F = \begin{pmatrix} 0 & 1 \cr 1 & 0\end{pmatrix}, \qquad \gamma = \begin{pmatrix} 1 & 0 \cr 0 & -1\end{pmatrix}.
 $$
Obviously $F$ is self-adjoint, squares to 1 and anti-commutes with $\gamma$, as required. As in \cite[Section~4.1.1]{HawLan:fred}, for all $a\in \cO(S^{2n}_q)\square_{\mathcal{O}(\mZt)}\cO(U(1))$,
 $$
 [F, \pi_\phi(a)] = \begin{pmatrix} 0 & -\pi_{\phi,+}(a) + \pi_{\phi,-}(a) \cr \pi_{\phi,+}(a) - \pi_{\phi,-}(a) & 0\end{pmatrix}.
 $$
 Note that $\pi_{\phi,+}- \pi_{\phi,-} =\pi_{\phi,+}\circ (\id -\nu)$, where $\nu$ is the $*$-algebra automorphism of  
  $\mathcal{A}^{2n}$
 given by
 $$
 \nu :\xi \longmapsto \xi, \quad  \zeta_i \longmapsto \left\{ \begin{array}{ll}
 \zeta_i & \mbox{if $i\neq n$} \\
-\zeta_n & \mbox{if $i= n$}\ . 
\end{array}
\right.
$$
This implies that $\pi_{\phi,+}(a) - \pi_{\phi,-}(a)$ is always a multiple of  $\pi_{\phi,+}(\zeta_n)$, which is a compact, trace-class operator.
Therefore, $[F, \pi_\phi(a)]$ is a compact, in fact trace-class,  operator on $\mathcal{H}_\phi$. Thus $(\mathcal{H}_\phi, F,\gamma)$ is a 1-summable Fredholm module over 
 $\mathcal{A}^{2n}$.
  This allows one to define a trace $\tau$ or the zero-component of the Chern character of $(\mathcal{H}_\phi, F,\gamma)$ by
 $$
 \tau(a) := \mathrm{Tr}(\gamma \pi_\phi(a)) =  \mathrm{Tr}(\pi_{\phi,+}(a) - \pi_{\phi,-}(a)).
 $$
 On the basis of $\mathcal{A}^{2n}$ given in the proof of Proposition~\ref{prop.present}~(2), $\zeta_0^{k}\zeta_1^{k_1}\cdots \zeta_n^{k_n}\zeta_1^{*l_1}\cdots \zeta_{n-1}^{*l_{n-1}}\xi^m$, the trace is non-zero (and given by a rational function of $q$ multiplied by $e^{i(k_n-2m)\phi}$) only when simultaneously $k_n$ is odd, $k=0$ and  $k_i=l_i$, for all $i <n$.

\begin{remark}\label{rem.lens}
Lemma~\ref{lemma.ident} can also be used to determine the algebras obtained as  prolongations of the $\cO(\Z_p)$-coaction on $\cO(S^{2n+1}_q)$, for all $p>1$.  $\cO(\Z_p)$ is a Hopf $*$-algebra  generated by $w$ subject to relations $w^p =1$ and $w^* = w^{p-1}$. The coaction of $\cO(\Z_p)$ on $\cO(S^{2n+1}_q)$ is given on generators by $z_i\mapsto z_i\otimes w$. The subalgebra of coaction invariants is known as the {\em quantum lens space} $\cO(L_q(p; \mathbf{1}))$ \cite{HonSzy:len}. The coaction of $\cO(\Z_p)$ on $\cO(S^{2n+1}_q)$ can be equivalently defined as the projection of the $\cO(U(1)$-coaction through the Hopf $*$-algebra map $\pi_p: \cO(U(1)) \to \cO(\Z_p)$,  $v\mapsto w$. Now, the combination of Lemma~\ref{lemma.ident} with the arguments of the proof of Proposition~\ref{prop.present}~(1) yields a $*$-algebra isomorphism
$$
\cO(S^{2n+1}_q)\square_{ \cO(\Z_p)} \cO(U(1)) \simeq \cO(S^{2n+1}_q)\otimes \cO(U(1)).   
$$
\hfill $\Diamond$
\end{remark}

  The Hopf algebra map $f_1\circ f_2: \cO(SU_q(2)) \to \cO(U(1))$, given on generators by $a\mapsto v$, $b\mapsto 0$ (see Section~\ref{sec.sph.proj}) has a unital $ \cO(U(1))$-bicolinear splitting \cite[p.\ 200]{BrzMaj:dif}, \cite[p.\ 257]{HajMaj:pro}
$$
\imath : v^n\longmapsto a^n, \qquad v^{*n}\longmapsto a^{*n}.
$$
This implies that $\cO(SU_q(2))$ is a left principal $\cO(U(1))$-comodule algebra.  By the identification
\begin{eqnarray*}
\cO(S^{m}_q)\square_{\mathcal{O}(\mZt)}\cO(SU_q(2)) &\simeq& \cO(S^{m}_q)\square_{\mathcal{O}(\mZt)}(\cO(U(1)) \square_{\mathcal{O}(U(1))}\cO(SU_q(2))) \\
&\simeq& (\cO(S^{m}_q)\square_{\mathcal{O}(\mZt)}\cO(U(1))) \square_{\mathcal{O}(U(1))}\cO(SU_q(2)),
\end{eqnarray*}
principal comodule algebras in \eqref{hier3} can also be understood as reductions of those in \eqref{hier2} by the Hopf ideal $\ker (f_1\circ f_2) \subset \cO(SU_q(2))$.  In particular, for $m=2,3$, the algebras $\cO(S^m_q)\square_{\mathcal{O}(\mZt)}\mathcal{O}(U(1))$ are {\em non-trivial}  principal $\cO(U(1)$-comodule algebras  obtained as reductions of {\em trivial} principal $\cO(SU_q(2))$-comodule algebras $\cO(S^m_q)\square_{\mathcal{O}(\mZt)}\mathcal{O}(SU_{q}(2))$.

  \section*{Acknowledgments}
  The authors are grateful to Piotr M.\ Hajac for many interesting discussions. BZ would like to thank the Department of Mathematics, Swansea University, for hospitality.
The results presented in this paper form a part of the project {\em Geometry and Symmetry of Quantum Spaces} 
PIRSES-GA-2008-230836. The research of BZ is also partially sponsored by the
Polish government matching grant: 1261/7.PR UE/2009/7.

\end{document}